\documentclass[11pt]{article}

\usepackage{amsmath}
\usepackage{amssymb}
\usepackage{amsthm}
\usepackage{indentfirst}
\usepackage{enumerate}
\usepackage{cite}
\usepackage{bigints}

\newtheorem{thm}{Theorem}[section]

\newtheorem{prop}{Proposition}[section]
\newtheorem{cor}{Corollary}[section]
\newtheorem{rem}{Remark}[section]

\usepackage[OT2,T1]{fontenc}
\input{cyracc.def}

\numberwithin{equation}{section}

\title{\textbf{Commutation Formulae With Respect to Non-Symmetric Affine Connection}}
\author{Nenad O. Vesi\'c\footnote{This work was supported by the Serbian Ministry of Education, Science and
Technological Development through Mathematical Institute of the
Serbian Academy of Sciences and Arts}, Du\v san J. Simjanovi\'c}
\date{}

\makeatletter
\def\maketag@@@#1{\hbox{\m@th\normalfont\normalsize#1}}
\makeatother

\begin{document}

  \maketitle

  \begin{abstract}
    Commutation formulae with respect to a non-symmetric affine
    connection are obtained in this paper. The components of
    commutation formulae in this paper are covariant derivatives of
    tensors with respect to symmetric and non-symmetric affine
    connection.
    \\[5pt]

    \textbf{Key words:} covariant derivative, commutation formula,
    linear independence\\[2pt]

    \textbf{$2010$ Math. Subj. Classification:} 53B05, 15A03
  \end{abstract}

  \section{Introduction}

  Identities of Ricci Type \cite{mik5,mincic2',mincic3,mincic4,mincic2,mincicnovi,mincvel,
  petrovic2019,petrovicvelimirovic2018,petrovicvelimirovic2019,sinjukov,stankoviczlatanovicvelimirovic2010,
  z4} are important for different researches in
  the fields of differential geometry and the corresponding
  applications.

  One curvature tensor of a symmetric affine
  connection space is obtained with respect to a symmetric affine
  connection \cite{mik5,sinjukov}. Many curvature tensors and
  curvature pseudotensors are founded with respect to a
  non-symmetric affine connection \cite{mincic2',mincic3,mincic4,mincic2,mincicnovi,mincvel,
  petrovic2019,petrovicvelimirovic2018,petrovicvelimirovic2019,stankoviczlatanovicvelimirovic2010,
  z4}. Curvature tensors and curvature pseudotensors are components
  of the curvature for the corresponding affine connection spaces.

  Our purpose is to obtain all identities of Ricci Type with respect to a non-symmetric affine connection in this
  paper. In this research, we will try to simplify the previously
  obtained identities.

  \subsection{Affine connection space}

  An $N$-dimensional manifold $\mathcal M_N$ equipped with an affine
  connection with torsion $\nabla$ is the generalized affine
  connection space $\mathbb{GA}_N$ \big(see \cite{eisNRG,mincic2',mincic3,mincic4,mincic2,mincicnovi,mincvel,
  petrovic2019,petrovicvelimirovic2018,petrovicvelimirovic2019,stankoviczlatanovicvelimirovic2010,
  z4}\big).

  The affine connection coefficients with respect to the affine
  connection (with torsion) $\nabla$ are $L^i_{jk}$, $L^i_{jk}\neq L^i_{kj}$. The
  symmetric and anti-symmetric parts of the affine connection
  coefficients $L^i_{jk}$ are

  \begin{equation}
  \begin{array}{ccc}
    L^i_{\underline{jk}}=\frac{1}2\big(L^i_{jk}+L^i_{kj}\big)&\mbox{and}&
    L^i_{\underset\vee{jk}}=\frac{1}2\big(L^i_{jk}-L^i_{kj}\big).
  \end{array}
  \label{eq:Lsimantisim}
  \end{equation}

  The double anti-symmetric parts $L^i_{\underset\vee{jk}}$ are the
  components of the torsion tensor for the affine connection space
  $\mathbb{GA}_N$.

  The symmetric parts $L^i_{\underline{jk}}$ satisfy the
  transformation rule

  \begin{equation}
    L^{i'}_{\underline{j'k'}}=x^{i'}_ix^j_{j'}x^k_{k'}L^i_{\underline{jk}}+
    x^{i'}_ix^i_{j'k'}.
  \end{equation}

  For this reason, the manifold $\mathcal M_N$ equipped with the
  symmetric affine connection $\overset0\nabla$ whose coefficients are
  $L^i_{\underline{jk}}$ is the associated (symmetric affine connection) space $\mathbb A_N$ of the
  space $\mathbb{GA}_N$ (see \cite{mik5,sinjukov}\big).

  Covariant derivatives are defined with respect to torsion-free
  affine connections \cite{mik5,sinjukov} and affine connections with torsion \cite{eisNRG,mincic2',mincic3,mincic4,mincic2,mincicnovi,mincvel,
  petrovic2019,petrovicvelimirovic2018,petrovicvelimirovic2019,stankoviczlatanovicvelimirovic2010,
  z4}. With
  respect to double covariant derivatives, corresponding commutation
  formulae are obtained. From the commutation formulae, the
  curvature tensors for the spaces $\mathbb A_N$ and $\mathbb{GA}_N$
  are founded.

  \subsection{About covariant derivatives}

  It exists one kind of covariant derivative with respect to the
  affine connection $\overset0\nabla$ \big(see
  \cite{mik5,sinjukov}\big)

  \begin{equation}
    a^{i_1\ldots i_p}_{j_1\ldots j_q|k}=a^{i_1\ldots i_p}_{j_1\ldots j_q,k}
    +\sum_{u=1}^p{L^{i_u}_{\underline{\alpha
    k}}a^{i_1\ldots i_{u-1}\alpha i_{u+1}\ldots i_p}
    _{j_1\ldots j_q}}-\sum_{v=1}^q{L^\alpha_{\underline{j_vk}}a^{i_1\ldots i_p}_{j_1\ldots j_{v-1}\alpha j_{v+1}\ldots j_q}},
    \label{eq:covderivativesimpq}
  \end{equation}

  \noindent for a tensor $\hat a$ of the type $(p,q)$ whose
  components are $a^{i_1\ldots i_p}_{j_1\ldots j_q}$ and the partial
  derivative $\partial/\partial x^k$ denoted by comma.

  It exists one Ricci-Type identity with respect to the covariant
  derivative given by the equation (\ref{eq:covderivativesimpq})

  \begin{equation}
    \aligned
    a^{i_1\ldots i_p}_{j_1\ldots j_q|m|n}-
    a^{i_1\ldots i_p}_{j_1\ldots j_q|n|m}&=\sum_{u=1}^p{a^{i_1\ldots
    i_{u-1}\alpha i_{u+1}\ldots i_p}_{j_1\ldots j_q}R^i_{\alpha
    mn}}-
    \sum_{v=1}^q{a^{i_1\ldots i_p}_{j_1\ldots j_{v-1}\alpha
    j_{v+1}\ldots j_q}R^\alpha_{j_vmn}},
    \endaligned\label{eq:RicciTypeIdSymm}
  \end{equation}

  \noindent for the components

  \begin{equation}
    R^i_{jmn}=L^i_{\underline{jm},n}-L^i_{\underline{jn},m}+
    L^\alpha_{\underline{jm}}L^i_{\underline{\alpha n}}-
    L^\alpha_{\underline{jn}}L^i_{\underline{\alpha m}},
    \label{eq:RAN}
  \end{equation}

  \noindent of the curvature tensor $\hat R$ of the type $(1,3)$ for
  the associated space $\mathbb A_N$.

  There are four kinds of covariant derivatives with respect to the
  affine connection with torsion $\nabla$ \big(see \cite{eisNRG,mincic2',mincic3,mincic4,mincic2,mincicnovi,mincvel,
  petrovic2019,petrovicvelimirovic2018,petrovicvelimirovic2019,stankoviczlatanovicvelimirovic2010,
  z4}\big)

  \begin{align}
    &a^{i_1\ldots i_p}_{j_1\ldots j_q\underset1|k}=
    a^{i_1\ldots i_p}_{j_1\ldots j_q,k}+\sum_{u=1}^p{L^{i_u}_{\alpha
    k}a^{i_1\ldots i_{u-1}\alpha i_{u+1}\ldots i_p}_{j_1\ldots j_q}}
    -\sum_{v=1}^q{L^\alpha_{j_vk}a^{i_1\ldots i_p}_{j_1\ldots j_{v-1}\alpha j_{v+1}\ldots
    j_q}},\label{eq:covderivativensim1pq}\\&
    a^{i_1\ldots i_p}_{j_1\ldots j_q\underset2|k}=
    a^{i_1\ldots i_p}_{j_1\ldots j_q,k}+\sum_{u=1}^p{L^{i_u}_{k\alpha}a^{i_1\ldots i_{u-1}\alpha i_{u+1}\ldots i_p}_{j_1\ldots j_q}}
    -\sum_{v=1}^q{L^\alpha_{kj_v}a^{i_1\ldots i_p}_{j_1\ldots j_{v-1}\alpha j_{v+1}\ldots
    j_q}},\label{eq:covderivativensim2pq}\\
    &a^{i_1\ldots i_p}_{j_1\ldots j_q\underset3|k}=
    a^{i_1\ldots i_p}_{j_1\ldots j_q,k}+\sum_{u=1}^p{L^{i_u}_{\alpha
    k}a^{i_1\ldots i_{u-1}\alpha i_{u+1}\ldots i_p}_{j_1\ldots j_q}}
    -\sum_{v=1}^q{L^\alpha_{kj_v}a^{i_1\ldots i_p}_{j_1\ldots j_{v-1}\alpha j_{v+1}\ldots
    j_q}},\label{eq:covderivativensim3pq}\\&
    a^{i_1\ldots i_p}_{j_1\ldots j_q\underset4|k}=
    a^{i_1\ldots i_p}_{j_1\ldots j_q,k}+\sum_{u=1}^p{L^{i_u}_{k\alpha}a^{i_1\ldots i_{u-1}\alpha i_{u+1}\ldots i_p}_{j_1\ldots j_q}}
    -\sum_{v=1}^q{L^\alpha_{j_vk}a^{i_1\ldots i_p}_{j_1\ldots j_{v-1}\alpha j_{v+1}\ldots
    j_q}}.\label{eq:covderivativensim4pq}
   \end{align}

   Let be $a^{i_1\ldots i_p}_{j_1\ldots j_q\underset0|k}\equiv a^{i_1\ldots i_p}_{j_1\ldots
   j_q|k}$. We will study the differences $a^{i_1\ldots i_p}_{j_1\ldots
   j_q\underset{v_1}|m\underset{w_1}|n}-a^{i_1\ldots i_p}_{j_1\ldots
   j_q\underset{v_2}|n\underset{w_2}|m}$,\linebreak
   $v_1,v_2,w_1,w_2\in\{0,1,2,3,4\}$, in this paper.

  \subsection{Motivation}

  It is obtained the Ricci-Type identity \cite{eisNRG,mincic2',mincic3,mincic4,mincic2,mincicnovi,mincvel,
  petrovic2019,petrovicvelimirovic2018,petrovicvelimirovic2019,stankoviczlatanovicvelimirovic2010,
  z4}

  \begin{equation}
    a^i_{j\underset1|m\underset1|n}-a^i_{j\underset1|n\underset1|m}=
    a^\alpha_j\underset1A{}^i_{\alpha
    mn}-a^i_\alpha\underset2A{}^\alpha_{jmn}+4a^i_{j<\underset\vee{mn}>}+
    4a^i_{j\leqslant\underset\vee{mn}\geqslant}+2L^\alpha_{\underset\vee{mn}}a^i_{j\underset1|\alpha},
    \label{eq:ricci1212}
  \end{equation}

  \noindent for

  \begin{align}
    &\underset1A{}^i_{jmn}=R^i_{jmn}+L^i_{\underset\vee{jm}|n}-
    L^i_{\underset\vee{jn}|m}-L^\alpha_{\underset\vee{jm}}L^i_{\underset\vee{\alpha
    n}}+L^\alpha_{\underset\vee{jn}}L^i_{\underset\vee{\alpha m}}-
    2L^\alpha_{\underline{jm}}L^i_{\underset\vee{\alpha n}}+
    L^\alpha_{\underline{jn}}L^i_{\underset\vee{\alpha
    m}},\label{eq:A1}\\
    &\underset2A{}^i_{jmn}=R^i_{jmn}+L^i_{\underset\vee{jm}|n}-
    L^i_{\underset\vee{jn}|m}-L^\alpha_{\underset\vee{jm}}L^i_{\underset\vee{\alpha
    n}}+L^\alpha_{\underset\vee{jn}}L^i_{\underset\vee{\alpha m}}-
    2L^\alpha_{\underset\vee{jm}}L^i_{\underline{\alpha n}}+
    L^\alpha_{\underline{jn}}L^i_{\underline{\alpha
    m}},\label{eq:A2}\\
    &a^i_{j<\underset\vee{mn}>}=\dfrac12L^i_{\underset\vee{\alpha
    m}}a^\alpha_{j,n}-\dfrac12L^i_{\underset\vee{\alpha
    n}}a^\alpha_{j,m}-
    \dfrac12L^\alpha_{\underset\vee{jm}}a^i_{\alpha,n}+
    \dfrac12L^\alpha_{\underset\vee{jn}}a^i_{\alpha,m},\label{eq:aij<veemn>}\\
    &a^i_{j\leqslant\underset\vee{mn}\geqslant}=
    \dfrac12a^\alpha_\beta\big(L^i_{m\beta}L^\alpha_{jn}-L^i_{n\beta}L^\alpha_{jm}-
    L^i_{\alpha m}L^\beta_{nj}+L^i_{\alpha
    n}L^\beta_{mj}\big).\label{eq:aij<=veemn=>}
  \end{align}

  The geometrical objects $\underset1A{}^i_{jmn}$ and
  $\underset2A{}^i_{jmn}$ are components of the curvature
  pseudotensors $\underset1{\hat A}$ and $\underset2{\hat A}$ of the
  type $(1,3)$. These objects are components of the curvature for
  the space $\mathbb{GA}_N$.

  In \cite{jacovder1}, and with respect to $L^i_{jk}=L^i_{\underline{jk}}+L^i_{\underset\vee{jk}}$,
   the equation (\ref{eq:ricci1212}) is simplified to

   \begin{equation}
     \aligned
     a^i_{j\underset1|m\underset2|n}-a^i_{j\underset1|n\underset2|m}&=2L^i_{\underset\vee{\alpha
     m}}a^\alpha_{j|n}-2L^i_{\underset\vee{\alpha
     n}}a^\alpha_{j|m}-2L^\alpha_{\underset\vee{jm}}a^i_{\alpha|n}+2L^\alpha_{\underset\vee{jn}}a^i_{\alpha|m}+2
     L^\alpha_{\underset\vee{mn}}a^i_{j|\alpha}\\&+
     a^\alpha_j\big(R^i_{\alpha mn}+L^i_{\underset\vee{\alpha m}|n}-
     L^i_{\underset\vee{\alpha n}|m}-
     L^\beta_{\underset\vee{\alpha m}}L^i_{\underset\vee{\beta n}}+
     L^\beta_{\underset\vee{\alpha n}}L^i_{\underset\vee{\beta m}}-
     2L^\beta_{\underset\vee{mn}}L^i_{\underset\vee{\beta\alpha}}\big)\\&
     -a^i_\alpha\big(R^\alpha_{jmn}+L^\alpha_{\underset\vee{jm}|n}-L^\alpha_{\underset\vee{jn}|m}-
     L^\beta_{\underset\vee{jm}}L^\alpha_{\underset\vee{\beta n}}+
     L^\beta_{\underset\vee{jn}}L^\alpha_{\underset\vee{\beta m}}-2
     L^\beta_{\underset\vee{mn}}L^\alpha_{\underset\vee{\beta
     j}}\big).
     \endaligned
     \tag{\ref{eq:ricci1212}'}\label{eq:ricci1212'}
   \end{equation}

  In \cite{jacovder1}, it is obtained the family of double covariant
  derivatives

  \begin{equation}
\aligned
  a^i_{j\underset v|m\underset
  w|n}&=a^i_{j|m|n}+c_vL^i_{\underset\vee{\alpha m}}a^\alpha_{j|n}+c_wL^i_{\underset\vee{\alpha n}}a^\alpha_{j|m}
  +d_vL^\alpha_{\underset\vee{jm}}a^i_{\alpha|n}+d_wL^\alpha_{\underset\vee{jn}}a^i_{\alpha|m}
  +d_wL^\alpha_{\underset\vee{mn}}a^i_{j|\alpha}\\&
  \aligned+
  a^\alpha_j\big(&c_vL^i_{\underset\vee{\alpha m}|n}+
  c_vc_wL^\beta_{\underset\vee{\alpha m}}L^i_{\underset\vee{\beta
  n}}+c_v(c_w+d_w)L^\beta_{\underset\vee{\alpha
  n}}L^i_{\underset\vee{\beta m}}
  -c_vd_wL^\beta_{\underset\vee{mn}}L^i_{\underset\vee{\beta\alpha}}\big)
  \endaligned\\&
  \aligned
  -a^i_\alpha\big(&-d_vL^\alpha_{\underset\vee{jm}|n}
  -d_v(c_w+d_w)L^\beta_{\underset\vee{jm}}L^\alpha_{\underset\vee{\beta n}}-d_vd_wL^\beta_{\underset\vee{jn}}L^\alpha_{\underset\vee{\beta
  m}}+d_vd_wL^\beta_{\underset\vee{mn}}L^\alpha_{\underset\vee{\beta
  j}}\big)
  \endaligned\\&+a^\alpha_\beta\big(c_wd_vL^\beta_{\underset\vee{jm}}L^i_{\underset\vee{\alpha
  n}}+c_vd_wL^\beta_{\underset\vee{jn}}L^i_{\underset\vee{\alpha
  m}}\big),
\endaligned\label{eq:doublecovariantderivative(1,1)}
\end{equation}

\noindent for $v,w\in\{0,1,2,3,4\}$.

When simplified the difference
$a^i_{j\underset{v_1}|m\underset{w_1}|n}-a^i_{j\underset
{v_2}|m\underset
  {w_2}|n}$, we proved the next theorem.

  \begin{thm}
    [First Ricci-Type Identities Theorem]\emph{\cite{jacovder1}}
    The family of identities of the Ricci Type with respect to a
    non-symmetric affine connection $\nabla$ is

      \begin{equation}
        \aligned
        a^i_{j\underset{v_1}|m\underset{w_1}|n}-
        a^i_{j\underset{v_2}|n\underset{w_2}|m}&=
        (c_{v_1}-c_{w_2})L^i_{\underset\vee{\alpha
        m}}a^\alpha_{j|n}+
        (c_{w_1}-c_{v_2})L^i_{\underset\vee{\alpha
        n}}a^\alpha_{j|m}+
        (d_{v_1}-d_{w_2})L^\alpha_{\underset\vee{jm}}a^i_{\alpha|n}\\&+
        (d_{w_1}-d_{v_2})L^\alpha_{\underset\vee{jn}}a^i_{\alpha|m}+
        (d_{w_1}+d_{w_2})L^\alpha_{\underset\vee{mn}}a^i_{j|\alpha}\\&
        \aligned+
  a^\alpha_j\Big\{R^i_{\alpha mn}&+c_{v_1}L^i_{\underset\vee{\alpha m}|n}-
  c_{v_2}L^i_{\underset\vee{\alpha n}|m}\\&+
  \big[c_{v_1}c_{w_1}-c_{v_2}(c_{w_2}+d_{w_2})\big]L^\beta_{\underset\vee{\alpha m}}L^i_{\underset\vee{\beta
  n}}\\&+\big[c_{v_1}(c_{w_1}+d_{w_1})-c_{v_2}c_{w_2}\big]L^\beta_{\underset\vee{\alpha n}}L^i_{\underset\vee{\beta
  m}}\\&
  -(c_{v_1}d_{w_1}+c_{v_2}d_{w_2})L^\beta_{\underset\vee{mn}}L^i_{\underset\vee{\beta\alpha}}\Big\}\endaligned\\&
        \aligned
  -a^i_\alpha\Big\{R^\alpha_{jmn}&-d_{v_1}L^\alpha_{\underset\vee{jm}|n}+
  d_{v_2}L^\alpha_{\underset\vee{jn}|m}
  \\&-\big[d_{v_1}(c_{w_1}+d_{w_1})-d_{v_2}d_{w_2}\big]L^\beta_{\underset\vee{jm}}L^\alpha_{\underset\vee{\beta n}}
  \\&-\big[d_{v_1}d_{w_1}-d_{v_2}(c_{w_2}+d_{w_2})\big]L^\beta_{\underset\vee{jn}}L^\alpha_{\underset\vee{\beta
  m}}\\&+(d_{v_1}d_{w_1}+d_{v_2}d_{w_2})L^\beta_{\underset\vee{mn}}L^\alpha_{\underset\vee{\beta
  j}}\Big\}\endaligned\\&
        +a^\alpha_\beta\big\{(c_{w_1}d_{v_1}-c_{v_2}d_{w_2})L^\beta_{\underset\vee{jm}}L^i_{\underset\vee{\alpha
        n}}+(c_{v_1}d_{w_1}-c_{w_2}d_{v_2})L^\beta_{\underset\vee{jn}}L^i_{\underset\vee{\alpha
        m}}\big\},
        \endaligned\label{eq:RicciTypeIDSfamily}
      \end{equation}

    \noindent for $v_1,v_2,w_1,w_2\in\{0,1,2,3,4\}$.\quad\qed
  \end{thm}

  It is obtained \cite{jacovder1} that the geometrical objects
  $a^i_{j\underset1|k}$, $a^i_{j\underset2|k}$,
  $a^i_{j\underset3|k}$ are linearly independent and that the
  geometrical objects $a^i_{j|k}$ and $a^i_{j\underset4|k}$ may be
  uniquely expressed in the terms of the first three kinds of
  covariant derivative.

  The purpose of this paper is to generalize the First Ricci-Type
  Identities Theorem in the sense of changing the summands
  $L^i_{\underset\vee{jk}}a^l_{s|r}$ with linear combinations of the
  geometrical objects $L^i_{\underset\vee{jk}}a^l_{s|r}\equiv
  L^i_{\underset\vee{jk}}a^l_{s\underset0|r}$,
  $L^i_{\underset\vee{jk}}a^l_{s\underset1|r}$,
  $L^i_{\underset\vee{jk}}a^l_{s\underset2|r}$,
  $L^i_{\underset\vee{jk}}a^l_{s\underset3|r}$,
  $L^i_{\underset\vee{jk}}a^l_{s\underset4|r}$.

  At the start of the research, we will
  prove that three of covariant derivatives $a^i_{j|k}$,
  $a^i_{j\underset1|k}$, $a^i_{j\underset2|k}$,
  $a^i_{j\underset3|k}$, $a^i_{j\underset4|k}$ are enough for all
  commutation formulae to be obtained.

  The next
  result of our research will be the commutation formulae with respect
  to double covariant derivatives of a tensor $\hat a$ of a type
  $(p,q)$, $p,q\in\mathbb N$.

%  At the end of the paper, we will
%  reduce the obtained commutation formulae with respect to the
%  tensors $\hat u$ and $\hat v$ of the types $(p,0)$ and $(0,q)$,
%  respectively.

  \section{Four plus one kinds of covariant derivatives}

  For the research in this paper, we need the next
  propositions.

  \begin{prop}
    The covariant derivatives given by the equations
    \emph{(\ref{eq:covderivativensim1pq},
    \ref{eq:covderivativensim2pq}, \ref{eq:covderivativensim3pq},
    \ref{eq:covderivativensim4pq})} and the covariant derivative
    given by the equation \emph{(\ref{eq:covderivativesimpq})}
    satisfy the equations

    \begin{align}
    &a^{i_1\ldots i_p}_{j_1\ldots j_q\underset1|k}=
    a^{i_1\ldots i_p}_{j_1\ldots j_q|k}+\sum_{u=1}^p{L^{i_u}_{\underset\vee{\alpha
    k}}a^{i_1\ldots i_{u-1}\alpha i_{u+1}\ldots i_p}_{j_1\ldots j_q}}
    -\sum_{v=1}^q{L^\alpha_{\underset\vee{j_vk}}a^{i_1\ldots i_p}_{j_1\ldots j_{v-1}\alpha j_{v+1}\ldots
    j_q}},\label{eq:ccovderivativensim1pq}\\&
    a^{i_1\ldots i_p}_{j_1\ldots j_q\underset2|k}=
    a^{i_1\ldots i_p}_{j_1\ldots j_q|k}-\sum_{u=1}^p{L^{i_u}_{\underset\vee{\alpha k}}a^{i_1\ldots i_{u-1}\alpha i_{u+1}\ldots i_p}_{j_1\ldots j_q}}
    +\sum_{v=1}^q{L^\alpha_{\underset\vee{j_vk}}a^{i_1\ldots i_p}_{j_1\ldots j_{v-1}\alpha j_{v+1}\ldots
    j_q}},\label{eq:ccovderivativensim2pq}\\
    &a^{i_1\ldots i_p}_{j_1\ldots j_q\underset3|k}=
    a^{i_1\ldots i_p}_{j_1\ldots j_q|k}+\sum_{u=1}^p{L^{i_u}_{\underset\vee{\alpha
    k}}a^{i_1\ldots i_{u-1}\alpha i_{u+1}\ldots i_p}_{j_1\ldots j_q}}
    +\sum_{v=1}^q{L^\alpha_{\underset\vee{j_vk}}a^{i_1\ldots i_p}_{j_1\ldots j_{v-1}\alpha j_{v+1}\ldots
    j_q}},\label{eq:ccovderivativensim3pq}\\&
    a^{i_1\ldots i_p}_{j_1\ldots j_q\underset4|k}=
    a^{i_1\ldots i_p}_{j_1\ldots j_q|k}-\sum_{u=1}^p{L^{i_u}_{\underset\vee{\alpha k}}a^{i_1\ldots i_{u-1}\alpha i_{u+1}\ldots i_p}_{j_1\ldots j_q}}
    -\sum_{v=1}^q{L^\alpha_{\underset\vee{j_vk}}a^{i_1\ldots i_p}_{j_1\ldots j_{v-1}\alpha j_{v+1}\ldots
    j_q}}.\label{eq:ccovderivativensim4pq}
   \end{align}
  \end{prop}

  \begin{rem}\label{remcovderivative}
    With respect to the equations
    \emph{(\ref{eq:covderivativesimpq},
    \ref{eq:ccovderivativensim1pq}--\ref{eq:ccovderivativensim4pq})},
    we obtain

    \begin{equation}
      a^{i_1\ldots i_p}_{j_1\ldots j_q\underset z|k}=
    a^{i_1\ldots i_p}_{j_1\ldots j_q|k}+c_z\sum_{u=1}^p{L^{i_u}_{\underset\vee{\alpha k}}a^{i_1\ldots i_{u-1}\alpha i_{u+1}\ldots i_p}_{j_1\ldots j_q}}
    +d_z\sum_{v=1}^q{L^\alpha_{\underset\vee{j_vk}}a^{i_1\ldots i_p}_{j_1\ldots j_{v-1}\alpha j_{v+1}\ldots
    j_q}},\label{eq:a=a+cz+dz}
    \end{equation}

    \noindent for $z=0,\ldots,4$,
    and the corresponding coefficients
    $c_0=d_0=0$, $c_1=1$, $c_2=-1$, $c_3=1$, $c_4=-1$, $d_1=-1$,
    $d_2=1$, $d_3=1$, $d_4=-1$.
  \end{rem}

  Let us obtain the commutation formulae with respect to covariant
  derivatives of tensors $\hat a$ of the type $(1,1)$, $\hat u$ of the type $(1,0)$ and
  $\hat v$ of the type $(0,1)$.

  \begin{prop}
    For a tensor $\hat a$ of the type $(1,1)$, three of the geometrical objects $a^i_{j\underset0|k}$,
    $a^i_{j\underset1|k}$, $a^i_{j\underset2|k}$,
    $a^i_{j\underset3|k}$, $a^i_{j\underset4|k}$ are linearly
    independent.
  \end{prop}

  \begin{proof}
    With respect to the equation (\ref{eq:a=a+cz+dz}), the number of
    linearly independent geometrical objects $a^i_{j\underset0|k}$,
    $a^i_{j\underset1|k}$, $a^i_{j\underset2|k}$,
    $a^i_{j\underset3|k}$, $a^i_{j\underset4|k}$ is equal to the
    rank of the matrix

    \begin{equation*}
      M=\left[
      \begin{array}{ccc}
        1&0&0\\
        1&1&-1\\
        1&-1&1\\
        1&1&1\\
        1&-1&-1
      \end{array}
      \right].
    \end{equation*}

    Because $Rank(M)=3$, three of the geometrical objects $a^i_{j\underset0|k}$,
    $a^i_{j\underset1|k}$, $a^i_{j\underset2|k}$,
    $a^i_{j\underset3|k}$, $a^i_{j\underset4|k}$ are linearly
    independent.
  \end{proof}

  \begin{cor}
    For a tensor $\hat u$ of the type $(1,0)$,
    two of the geometrical objects $u^i_{\underset0|k}$,
    $u^i_{\underset 1|k}\equiv u^i_{\underset3|k}$,
    $u^i_{\underset2| k}\equiv u^i_{\underset4|k}$ are linearly
    independent.

    For a tensor $\hat v$ of the type $(0,1)$, two of the
    geometrical objects $v_{j\underset0|k}$, $v_{j\underset1|k}\equiv
    v_{i\underset4|k}$, $v_{j\underset2|k}\equiv v_{j\underset3|k}$
    are linearly independent. \quad\qed
  \end{cor}

  \begin{cor}\label{cor:LinIndependentCovDer}
    The triples

    \begin{equation}
      \begin{array}{llll}
        \underset1{\mathcal A}:\left\{
        \begin{array}{l}
          a^i_{j\underset1|k},\\
          a^i_{j\underset2|k},\\
          a^i_{j\underset3|k},
        \end{array}
        \right.&\underset2{\mathcal A}:\left\{
        \begin{array}{l}
          a^i_{j\underset1|k},\\
          a^i_{j\underset2|k},\\
          a^i_{j\underset4|k},
        \end{array}
        \right.&\underset3{\mathcal A}:\left\{
        \begin{array}{l}
          a^i_{j\underset1|k},\\
          a^i_{j\underset3|k},\\
          a^i_{j\underset4|k},
        \end{array}
        \right.&\underset4{\mathcal A}:\left\{
        \begin{array}{l}
          a^i_{j\underset2|k},\\
          a^i_{j\underset3|k},\\
          a^i_{j\underset4|k},
        \end{array}
        \right.\\\\
        \underset5{\mathcal A}:\left\{
        \begin{array}{l}
          a^i_{j\underset0|k},\\
          a^i_{j\underset1|k},\\
          a^i_{j\underset3|k},
        \end{array}
        \right.&\underset6{\mathcal A}:\left\{
        \begin{array}{l}
          a^i_{j\underset0|k},\\
          a^i_{j\underset1|k},\\
          a^i_{j\underset4|k},
        \end{array}
        \right.&\underset7{\mathcal A}:\left\{
        \begin{array}{l}
          a^i_{j\underset0|k},\\
          a^i_{j\underset2|k},\\
          a^i_{j\underset3|k},
        \end{array}
        \right.&\underset8{\mathcal A}:\left\{
        \begin{array}{l}
          a^i_{j\underset0|k},\\
          a^i_{j\underset2|k},\\
          a^i_{j\underset4|k},
        \end{array}
        \right.
      \end{array}\label{eq:linindcovderA}
    \end{equation}

    \noindent are triples of linearly independent geometrical
    objects $a^i_{j\underset z|k}$, $z=0,\ldots,4$.

    The pairs

    \begin{footnotesize}
    \begin{equation}
      \begin{array}{ccccccc}
        \underset1{\mathcal U}:
        \left\{
        \begin{array}{l}
          u^i_{\underset0|k},\\
          u^i_{\underset1|k},
        \end{array}
        \right.&
        \underset2{\mathcal U}:
        \left\{
        \begin{array}{l}
          u^i_{\underset0|k},\\
          u^i_{\underset2|k},
        \end{array}
        \right.&
        \underset3{\mathcal U}:
        \left\{
        \begin{array}{l}
          u^i_{\underset1|k},\\
          u^i_{\underset2|k},
        \end{array}
        \right.&\mbox{\normalsize and}&
        \underset1{\mathcal V}:
        \left\{
        \begin{array}{l}
          v_{j\underset0|k},\\
          v_{j\underset1|k},
        \end{array}
        \right.&
        \underset2{\mathcal V}:
        \left\{
        \begin{array}{l}
          v_{j\underset0|k},\\
          v_{j\underset2|k},
        \end{array}
        \right.&
        \underset3{\mathcal V}:
        \left\{
        \begin{array}{l}
          v_{j\underset1|k},\\
          v_{j\underset2|k},
        \end{array}
        \right.
      \end{array}
    \end{equation}
    \end{footnotesize}

    \noindent are pairs of linearly independent geometrical objects
    $u^i_{\underset z|k}$, $v_{j\underset z|k}$, for
    $z=0,\ldots,4$.\quad\qed
  \end{cor}

  \section{Identities of Ricci Type with respect to tensor $\hat a$ of type $(1,1)$}

  Let us generalize the First Ricci-Type Identities Theorem.

  \begin{thm}[Second Ricci-Type Identities Theorem]
    Let be

    \begin{align}
      &X^i_{jk}=\rho^1_0a^i_{j|k}+\rho^1_1a^i_{j\underset1|k}+
      \rho^1_2a^i_{j\underset2|k}+\rho^1_3a^i_{j\underset3|k}+
      \rho^1_4a^i_{j\underset4|k},\label{eq:Xsecondthm}\\\displaybreak[0]
    &Y^i_{jk}=\rho^2_0a^i_{j|k}+\rho^2_1a^i_{j\underset1|k}+
      \rho^2_2a^i_{j\underset2|k}+\rho^2_3a^i_{j\underset3|k}+
      \rho^2_4a^i_{j\underset4|k},\label{eq:Ysecondthm}\\\displaybreak[0]
      &Z^i_{jk}=\rho^3_0a^i_{j|k}+\rho^3_1a^i_{j\underset1|k}+
      \rho^3_2a^i_{j\underset2|k}+\rho^3_3a^i_{j\underset3|k}+
      \rho^3_4a^i_{j\underset4|k},\label{eq:Zsecondthm}\\\displaybreak[0]
      &U^i_{jk}=\rho^4_0a^i_{j|k}+\rho^4_1a^i_{j\underset1|k}+
      \rho^4_2a^i_{j\underset2|k}+\rho^4_3A^i_{j\underset3|k}+
      \rho^4_4a^i_{j\underset4|k},\label{eq:Usecondthm}\\\displaybreak[0]
      &V^i_{jk}=\rho^5_0a^i_{j|k}+\rho^5_1a^i_{j\underset1|k}+
      \rho^5_2a^i_{j\underset2|k}+\rho^5_3a^i_{j\underset3|k}+
      \rho^5_4a^i_{j\underset4|k},\label{eq:Vsecondthm}
    \end{align}

    \noindent for a tensor $\hat a$ of the type $(1,1)$ and
     scalars
    $\rho^z_0,\rho^z_1,\rho^z_2,\rho^z_3,\rho^z_4$,
    $z\in\{1,\ldots,5\}$,
    $\rho^z_0+\rho^z_1+\rho^z_2+\rho^z_3+\rho^z_4=1$.

    It holds the equation

    \begin{equation}
      \aligned
      a^i_{j\underset{v_1}|m\underset{w_1}|n}-
      a^i_{j\underset{v_2}|n\underset{w_2}|m}&=(c_{v_1}-c_{w_2})L^i_{\underset\vee{\alpha
      m}}X^\alpha_{jn}+(c_{w_1}-c_{v_2})L^i_{\underset\vee{\alpha
      n}}Y^\alpha_{jm}+(d_{v_1}-d_{w_2})L^\alpha_{\underset\vee{jm}}Z^i_{\alpha
      n}\\&+(d_{w_1}-d_{v_2})L^\alpha_{\underset\vee{jn}}U^i_{\alpha
      m}+(d_{w_1}+d_{w_2})L^\alpha_{\underset\vee{mn}}V^i_{j\alpha}\\&
      \aligned+a^\alpha_j\big\{R^i_{\alpha
      mn}&+c_{v_1}L^i_{\underset\vee{\alpha
      m}|n}-c_{v_2}L^i_{\underset\vee{\alpha
      n}|m}\\&+p_1L^\beta_{\underset\vee{\alpha
      m}}L^i_{\underset\vee{\beta n}}
      +p_2L^\beta_{\underset\vee{\alpha n}}L^i_{\underset\vee{\beta
      m}}+p_3L^\beta_{\underset\vee{mn}}L^i_{\underset\vee{\beta\alpha}}\big\}
      \endaligned\\
      &\aligned
      -a^i_\alpha\big\{R^\alpha_{jmn}&-d_{v_1}L^\alpha_{\underset\vee{jm}|n}+d_{v_2}L^\alpha_{\underset\vee{jn}|m}\\&+
      q_1L^\beta_{\underset\vee{jm}}L^\alpha_{\underset\vee{\beta n}}+
      q_2L^\beta_{\underset\vee{jn}}L^\alpha_{\underset\vee{\beta m}}+
      q_3L^\beta_{\underset\vee{mn}}L^\alpha_{\underset\vee{\beta
      j}}\big\}
      \endaligned\\&+
      a^\alpha_\beta\big\{
      r_1L^\beta_{\underset\vee{jm}}L^i_{\underset\vee{\alpha n}}+
      r_2L^\beta_{\underset\vee{jn}}L^i_{\underset\vee{\alpha m}}
      \big\},
      \endaligned\label{eq:RicciTypeIDSfamilygen}
    \end{equation}

    \noindent where
    \begin{align}
      &p_1=c_{v_1}c_{w_1}-c_{v_2}(c_{w_2}+d_{w_2})-(c_{w_1}-c_{v_2})(\rho^2_1-\rho^2_2+\rho^2_3-\rho^2_4),
      \label{eq:srtithma1}\\\displaybreak[0]
      &p_2=c_{v_1}(c_{w_1}+d_{w_1})-c_{v_2}c_{w_2}-(c_{v_1}-c_{w_2})(\rho^1_1-\rho^1_2+\rho^1_3-\rho^1_4),
      \label{eq:srtithma2}\\\displaybreak[0]
      &p_3=-c_{v_1}d_{w_1}-c_{v_2}d_{w_2}+(d_{w_1}+d_{w_2})(\rho^5_1-\rho^5_2+\rho^5_3-\rho^5_4),
      \label{eq:srtithma3}\\\displaybreak[0]
      &q_1=-d_{v_1}(c_{w_1}+d_{w_1})+d_{v_2}d_{w_2}-(d_{v_1}-d_{w_2})(\rho^3_1-\rho^3_2-\rho^3_3+\rho^3_4),
      \label{eq:srtithmb1}\\\displaybreak[0]
      &q_2=-d_{v_1}d_{w_1}+d_{v_2}(c_{w_2}+d_{w_2})-(d_{w_1}-d_{v_2})(\rho^4_1-\rho^4_2-\rho^4_3+\rho^4_4),
      \label{eq:srtithmb2}\\\displaybreak[0]
      &q_3=d_{v_1}d_{w_1}+d_{v_2}d_{w_2}-(d_{w_1}+d_{w_2})(\rho^5_1-\rho^5_2-\rho^5_3+\rho^5_4),
      \label{eq:srtithmb3}\\\displaybreak[0]
      &\aligned
      r_1&=c_{w_1}d_{v_1}-c_{v_2}d_{w_2}+(c_{w_1}-c_{v_2})(\rho^2_1-\rho^2_2-\rho^2_3+\rho^2_4)\\&-
      (d_{v_1}-d_{w_2})(\rho^3_1-\rho^3_2+\rho^3_3-\rho^3_4),
      \endaligned\label{eq:srtithml1}\\\displaybreak[0]
      &\aligned
      r_2&=c_{v_1}d_{w_1}-c_{w_2}d_{v_2}+(c_{v_1}-c_{w_2})(\rho^1_1-\rho^1_2-\rho^1_3+\rho^1_4)\\&-
      (d_{w_1}-d_{v_2})(\rho^4_1-\rho^4_2+\rho^4_3-\rho^4_4).
      \endaligned\label{eq:srtithml2}
    \end{align}
  \end{thm}

  \begin{proof}
    We get

    \begin{align}
      &L^i_{\underset\vee{\alpha
      m}}X^\alpha_{jn}=L^i_{\underset\vee{\alpha m}}a^\alpha_{j|n}+
      (\rho^1_1-\rho^1_2+\rho^1_3-\rho^1_4)L^i_{\underset\vee{\beta
      m}}L^\beta_{\underset\vee{\alpha n}}a^\alpha_j-
      (\rho^1_1-\rho^1_2-\rho^1_3+\rho^1_4)L^i_{\underset\vee{\alpha
      m}}L^\beta_{\underset\vee{jn}}a^\alpha_\beta,\label{eq:srtithmproofx}\\\displaybreak[0]
      &L^i_{\underset\vee{\alpha
      n}}Y^\alpha_{jm}=L^i_{\underset\vee{\alpha n}}a^\alpha_{j|m}+
      (\rho^2_1-\rho^2_2+\rho^2_3-\rho^2_4)L^i_{\underset\vee{\beta
      n}}L^\beta_{\underset\vee{\alpha m}}a^\alpha_j-
      (\rho^2_1-\rho^2_2-\rho^2_3+\rho^2_4)L^i_{\underset\vee{\alpha
      n}}L^\beta_{\underset\vee{jm}}a^\alpha_\beta,\label{eq:srtithmproofy}\\\displaybreak[0]
      &L^\alpha_{\underset\vee{jm}}Z^i_{\alpha n}=
      L^\alpha_{\underset\vee{jm}}a^i_{\alpha|n}+
      (\rho^3_1-\rho^3_2+\rho^3_3-\rho^3_4)L^\beta_{\underset\vee{jm}}L^i_{\underset\vee{\alpha n}}a^\alpha_\beta-
      (\rho^3_1-\rho^3_2-\rho^3_3+\rho^3_4)L^\beta_{\underset\vee{jm}}
      L^\alpha_{\underset\vee{\beta n}}a^i_\alpha,\label{eq:srtithmproofz}\\\displaybreak[0]
      &L^\alpha_{\underset\vee{jn}}U^i_{\alpha m}=L^\alpha_{\underset\vee{jn}}a^i_{\alpha|m}+
      (\rho^4_1-\rho^4_2+\rho^4_3-\rho^4_4)L^\beta_{\underset\vee{jn}}L^i_{\underset\vee{\alpha m}}a^\alpha_\beta-
      (\rho^4_1-\rho^4_2-\rho^4_3+\rho^4_4)L^\beta_{\underset\vee{jn}}
      L^\alpha_{\underset\vee{\beta m}}a^i_\alpha,\label{eq:srtithmproofu}\\\displaybreak[0]
      &L^\alpha_{\underset\vee{mn}}V^i_{j\alpha}=L^\alpha_{\underset\vee{mn}}a^i_{j|\alpha}+
      (\rho^5_1-\rho^5_2+\rho^5_3-\rho^5_4)L^\beta_{\underset\vee{mn}}L^i_{\underset\vee{\alpha\beta}}a^\alpha_j-
      (\rho^5_1-\rho^5_2-\rho^5_3+\rho^5_4)L^\beta_{\underset\vee{mn}}
      L^\alpha_{\underset\vee{\beta j}}a^i_\alpha,\label{eq:srtithmproofv}
    \end{align}

    After expressing the terms

    \begin{equation*}
      \begin{array}{ccc}
      (c_{v_1}-c_{w_2})L^i_{\underset\vee{\alpha
      m}}a^\alpha_{j|n},&(c_{w_1}-c_{v_2})L^i_{\underset\vee{\alpha
      n}}a^\alpha_{j|m},&(d_{v_1}-d_{w_2})L^\alpha_{\underset\vee{jm}}a^i_{\alpha
      |n},\\
      \multicolumn{3}{c}{\begin{array}{cc}(d_{w_1}-d_{v_2})L^\alpha_{\underset\vee{jn}}a^i_{\alpha|
      m},&(d_{w_1}+d_{w_2})L^\alpha_{\underset\vee{mn}}a^i_{j|\alpha},
      \end{array}}\end{array}
    \end{equation*}

    \noindent with respect to the equalities
    (\ref{eq:srtithmproofx}--\ref{eq:srtithmproofv}) and
    substituting them into the equation (\ref{eq:RicciTypeIDSfamily}),
    one confirms the validity of the equation
    (\ref{eq:RicciTypeIDSfamilygen}).
  \end{proof}

The next equalities are satisfied

\begin{align}
  &\rho^z_1-\rho^z_2+\rho^z_3-\rho^z_4=(-1)^{1-1}\rho^z_1+(-1)^{2-1}\rho^z_2
  +(-1)^{3-1}\rho^z_3+\rho^{4-1}\rho^z_4,\label{eq:rho+-+-}\\\displaybreak[0]
  &\rho^z_1-\rho^z_2-\rho^z_3+\rho^z_4=
  (-1)^{\big\lfloor\frac12\big\rfloor}\rho^z_1+
  (-1)^{\big\lfloor\frac22\big\rfloor}\rho^z_2+
  (-1)^{\big\lfloor\frac32\big\rfloor}\rho^z_3+
  (-1)^{\big\lfloor\frac42\big\rfloor}\rho^z_4,\label{eq:rho+--+}
\end{align}

\noindent $z=1,\ldots,5$, for the floor function $\lfloor x\rfloor$
(the function that takes as input a real number $x$ and gives the
greatest integer less than or equal to $x$  as output).

Let be $\{n_1,n_2,n_3,n_4\}=\{1,2,3,4\}$. With respect to the
Proposition \ref{cor:LinIndependentCovDer}, we conclude that it is
enough to consider the case of $\rho^z_0=\rho^z_{n_4}=0$,
${n_4}\in\{1,2,3,4\}\backslash\{n_1,n_2,n_3\}$.

For integers $n_1,n_2,n_3$, $1\leq n_1<n_2<n_3\leq 4$, and with
respect to the equations (\ref{eq:rho+-+-}, \ref{eq:rho+--+})
substituted into the equation (\ref{eq:RicciTypeIDSfamilygen}), the
next theorem holds.

\begin{thm}[$n_1-n_2-n_3$-Second Ricci-Type Identities Theorem]
  Let be

  \begin{equation}
    \begin{array}{cc}
      \tilde X{}^i_{jk}=\rho^1_{n_1}a^i_{j\underset{n_1}|k}+
      \rho^1_{n_2}a^i_{j\underset{n_2}|k}+\rho^1_{n_3}a^i_{j\underset{n_3}|k},&
      \tilde Y{}^i_{jk}=\rho^2_{n_1}a^i_{j\underset{n_1}|k}+
      \rho^2_{n_2}a^i_{j\underset{n_2}|k}+\rho^2_{n_3}a^i_{j\underset{n_3}|k},\\
      \tilde Z{}^i_{jk}=\rho^3_{n_1}a^i_{j\underset{n_1}|k}+
      \rho^3_{n_2}a^i_{j\underset{n_2}|k}+\rho^3_{n_3}a^i_{j\underset{n_3}|k},&
      \tilde U{}^i_{jk}=\rho^4_{n_1}a^i_{j\underset{n_1}|k}+
      \rho^4_{n_2}a^i_{j\underset{n_2}|k}+\rho^4_{n_3}a^i_{j\underset{n_3}|k},\\
      \multicolumn{2}{c}{
      \tilde V{}^i_{jk}=\rho^5_{n_1}a^i_{j\underset{n_1}|k}+
      \rho^5_{n_2}a^i_{j\underset{n_2}|k}+\rho^5_{n_3}a^i_{j\underset{n_3}|k},}
    \end{array}
  \end{equation}

  \noindent for the tensor $\hat a$ of the type $(1,1)$.

  It holds the equation

  \begin{equation}
      \aligned
      a^i_{j\underset{v_1}|m\underset{w_1}|n}-
      a^i_{j\underset{v_2}|n\underset{w_2}|m}&=(c_{v_1}-c_{w_2})L^i_{\underset\vee{\alpha
      m}}\tilde X{}^\alpha_{jn}+(c_{w_1}-c_{v_2})L^i_{\underset\vee{\alpha
      n}}\tilde Y{}^\alpha_{jm}+(d_{v_1}-d_{w_2})L^\alpha_{\underset\vee{jm}}\tilde Z{}^i_{\alpha
      n}\\&+(d_{w_1}-d_{v_2})L^\alpha_{\underset\vee{jn}}\tilde U{}^i_{\alpha
      m}+(d_{w_1}+d_{w_2})L^\alpha_{\underset\vee{mn}}\tilde V{}^i_{j\alpha}\\&
      \aligned+a^\alpha_j\big\{R^i_{\alpha
      mn}&+c_{v_1}L^i_{\underset\vee{\alpha
      m}|n}-c_{v_2}L^i_{\underset\vee{\alpha
      n}|m}\\&+\tilde p{}_1L^\beta_{\underset\vee{\alpha
      m}}L^i_{\underset\vee{\beta n}}
      +\tilde p{}_2L^\beta_{\underset\vee{\alpha n}}L^i_{\underset\vee{\beta
      m}}+\tilde p{}_3L^\beta_{\underset\vee{mn}}L^i_{\underset\vee{\beta\alpha}}\big\}
      \endaligned\\
      &\aligned
      -a^i_\alpha\big\{R^\alpha_{jmn}&-d_{v_1}L^\alpha_{\underset\vee{jm}|n}+d_{v_2}L^\alpha_{\underset\vee{jn}|m}\\&+
      \tilde q{}_1L^\beta_{\underset\vee{jm}}L^\alpha_{\underset\vee{\beta n}}+
      \tilde q{}_2L^\beta_{\underset\vee{jn}}L^\alpha_{\underset\vee{\beta m}}+
      \tilde q{}_3L^\beta_{\underset\vee{mn}}L^\alpha_{\underset\vee{\beta
      j}}\big\}
      \endaligned\\&+
      a^\alpha_\beta\big\{
      \tilde r{}_1L^\beta_{\underset\vee{jm}}L^i_{\underset\vee{\alpha n}}+
      \tilde r{}_2L^\beta_{\underset\vee{jn}}L^i_{\underset\vee{\alpha m}}
      \big\},
      \endaligned\label{eq:RicciTypeIDSfamilygen}
    \end{equation}

    \noindent where
    \begin{align}
      &\tilde p_1=c_{v_1}c_{w_1}-c_{v_2}(c_{w_2}+d_{w_2})-(c_{w_1}-c_{v_2})\big((-1)^{n_1-1}\rho^2_{n_1}+(-1)^{n_2-1}\rho^2_{n_2}
      +(-1)^{n_3-1}\rho^2_{n_3}\big),
      \label{eq:srtithma1tilde}\\\displaybreak[0]
      &\tilde p{}_2=c_{v_1}(c_{w_1}+d_{w_1})-c_{v_2}c_{w_2}-(c_{v_1}-c_{w_2})\big((-1)^{n_1-1}\rho^1_{n_1}+(-1)^{n_2-1}\rho^1_{n_2}
      +(-1)^{n_3-1}\rho^1_{n_3}\big),
      \label{eq:srtithma2tilde}\\\displaybreak[0]
      &\tilde p{}_3=-c_{v_1}d_{w_1}-c_{v_2}d_{w_2}+(d_{w_1}+d_{w_2})\big((-1)^{n_1-1}\rho^5_{n_1}+(-1)^{n_2-1}\rho^5_{n_2}
      +(-1)^{n_3-1}\rho^5_{n_3}\big),
      \label{eq:srtithma3tilde}\\\displaybreak[0]
      &\tilde q{}_1=-d_{v_1}(c_{w_1}+d_{w_1})+d_{v_2}d_{w_2}-(d_{v_1}\!-\!d_{w_2})
      \big((-1)^{\big\lfloor\frac{n_1}2\big\rfloor}\rho^3_{n_1}\!+\!(-1)^{\big\lfloor\frac{n_2}2\big\rfloor}\rho^3_{n_2}
      \!+\!(-1)^{\big\lfloor\frac{n_3}2\big\rfloor}\rho^3_{n_3}\big),
      \label{eq:srtithmb1tilde}\\\displaybreak[0]
      &\tilde q{}_2=-d_{v_1}d_{w_1}+d_{v_2}(c_{w_2}+d_{w_2})-
      (d_{w_1}-d_{v_2})\big((-1)^{\big\lfloor\frac{n_1}2\big\rfloor}\rho^4_1+(-1)^{\big\lfloor\frac{n_2}2\big\rfloor}\rho^4_{n_2}
      +(-1)^{\big\lfloor\frac{n_3}2\big\rfloor}\rho^4_{n_3}\big),
      \label{eq:srtithmb2tilde}\\\displaybreak[0]
      &\tilde q{}_3=d_{v_1}d_{w_1}+d_{v_2}d_{w_2}-(d_{w_1}+d_{w_2})
      \big((-1)^{\big\lfloor\frac{n_1}2\big\rfloor}\rho^5_{n_1}+(-1)^{\big\lfloor\frac{n_2}2\big\rfloor}\rho^5_{n_2}
      +(-1)^{\big\lfloor\frac{n_3}2\big\lfloor}\rho^5_{n_3}\big),
      \label{eq:srtithmb3tilde}\\\displaybreak[0]
      &\aligned
      \tilde r{}_1&=c_{w_1}d_{v_1}-c_{v_2}d_{w_2}+(c_{w_1}-c_{v_2})
      \big((-1)^{\big\lfloor\frac{n_1}2\big\rfloor}\rho^2_{n_1}+(-1)^{\big\lfloor\frac{n_2}2\big\rfloor}\rho^2_{n_2}
      +(-1)^{\big\lfloor\frac{n_3}2\big\rfloor}\rho^2_{n_3}\big)\\&-
      (d_{v_1}-d_{w_2})\big((-1)^{n_1-1}\rho^3_{n_1}+(-1)^{n_2-1}\rho^3_{n_2}+(-1)^{n_3-1}\rho^3_{n_3}\big),
      \endaligned\label{eq:srtithml1tilde}\\\displaybreak[0]
      &\aligned
      \tilde r{}_2&=c_{v_1}d_{w_1}-c_{w_2}d_{v_2}+(c_{v_1}-c_{w_2})
      \big((-1)^{\big\lfloor\frac{n_1}2\big\rfloor}\rho^1_{n_1}+(-1)^{\big\lfloor\frac{n_2}2\big\rfloor}\rho^1_{n_2}
      +(-1)^{\big\lfloor\frac{n_3}2\big\rfloor}\rho^1_{n_3}\big)\\&-
      (d_{w_1}-d_{v_2})
      \big((-1)^{n_1-1}\rho^4_{n_1}+(-1)^{n_2-1}\rho^4_{n_2}+(-1)^{n_3-1}\rho^4_{n_3}\big),
      \endaligned\label{eq:srtithml2}
    \end{align}

    \noindent $\rho^z_{n_1}+\rho^z_{n_2}+\rho^z_{n_3}=1$,
    $z\in\{1,2,3,4,5\}$.\qed
\end{thm}

With respect to the Proposition \ref{cor:LinIndependentCovDer}, we
conclude that it is enough to consider the case of $\rho^z_{n_3}=0$,
$\rho^z_{n_4}=0$, $1\leq n_3<n_4\leq 4$,
$\{n_1,n_2\}=\{1,2,3,4\}\backslash\{n_3,n_4\}$,
$(n_1,n_2)\in\big\{(1,3),(1,4),(2,3),(2,4) \big\}$ as in the next
theorem.

\begin{thm}[$n_1-n_2$-Second Ricci-Type Identities Theorem]
  Let be

  \begin{equation}
    \begin{array}{cc}
      \tilde{\tilde X}{}^i_{jk}=\rho^1_{0}a^i_{j|k}+\rho^1_{n_1}a^i_{j\underset{n_1}|k}+
      \rho^1_{n_2}a^i_{j\underset{n_2}|k},&
      \tilde{\tilde Y}{}^i_{jk}=\rho^2_0a^i_{j|k}+\rho^2_{n_1}a^i_{j\underset{n_1}|k}+
      \rho^2_{n_2}a^i_{j\underset{n_2}|k},\\
      \tilde{\tilde Z}{}^i_{jk}=\rho^3_0a^i_{j|k}+\rho^3_{n_1}a^i_{j\underset{n_1}|k}+
      \rho^3_{n_2}a^i_{j\underset{n_2}|k},&
      \tilde{\tilde U}{}^i_{jk}=\rho^4_0a^i_{j|k}+\rho^4_{n_1}a^i_{j\underset{n_1}|k}+
      \rho^4_{n_2}a^i_{j\underset{n_2}|k},\\
      \multicolumn{2}{c}{
      \tilde{\tilde V}{}^i_{jk}=\rho^5_0a^i_{j|k}+\rho^5_{n_1}a^i_{j\underset{n_1}|k}+
      \rho^5_{n_2}a^i_{j\underset{n_2}|k},}
    \end{array}
  \end{equation}

  \noindent for the tensor $\hat a$ of the type $(1,1)$.

  It holds the equation

  \begin{equation}
      \aligned
      a^i_{j\underset{v_1}|m\underset{w_1}|n}-
      a^i_{j\underset{v_2}|n\underset{w_2}|m}&=(c_{v_1}-c_{w_2})L^i_{\underset\vee{\alpha
      m}}\tilde{\tilde X}{}^\alpha_{jn}+(c_{w_1}-c_{v_2})L^i_{\underset\vee{\alpha
      n}}\tilde{\tilde Y}{}^\alpha_{jm}+(d_{v_1}-d_{w_2})L^\alpha_{\underset\vee{jm}}\tilde{\tilde Z}{}^i_{\alpha
      n}\\&+(d_{w_1}-d_{v_2})L^\alpha_{\underset\vee{jn}}\tilde{\tilde U}{}^i_{\alpha
      m}+(d_{w_1}+d_{w_2})L^\alpha_{\underset\vee{mn}}\tilde{\tilde V}{}^i_{j\alpha}\\&
      \aligned+a^\alpha_j\big\{R^i_{\alpha
      mn}&+c_{v_1}L^i_{\underset\vee{\alpha
      m}|n}-c_{v_2}L^i_{\underset\vee{\alpha
      n}|m}\\&+\tilde{\tilde p}{}_1L^\beta_{\underset\vee{\alpha
      m}}L^i_{\underset\vee{\beta n}}
      +\tilde{\tilde p}{}_2L^\beta_{\underset\vee{\alpha n}}L^i_{\underset\vee{\beta
      m}}+\tilde{\tilde p}{}_3L^\beta_{\underset\vee{mn}}L^i_{\underset\vee{\beta\alpha}}\big\}
      \endaligned\\
      &\aligned
      -a^i_\alpha\big\{R^\alpha_{jmn}&-d_{v_1}L^\alpha_{\underset\vee{jm}|n}+d_{v_2}L^\alpha_{\underset\vee{jn}|m}\\&+
      \tilde{\tilde q}{}_1L^\beta_{\underset\vee{jm}}L^\alpha_{\underset\vee{\beta n}}+
      \tilde{\tilde q}{}_2L^\beta_{\underset\vee{jn}}L^\alpha_{\underset\vee{\beta m}}+
      \tilde{\tilde q}{}_3L^\beta_{\underset\vee{mn}}L^\alpha_{\underset\vee{\beta
      j}}\big\}
      \endaligned\\&+
      a^\alpha_\beta\big\{
      \tilde{\tilde r}{}_1L^\beta_{\underset\vee{jm}}L^i_{\underset\vee{\alpha n}}+
      \tilde{\tilde r}{}_2L^\beta_{\underset\vee{jn}}L^i_{\underset\vee{\alpha m}}
      \big\},
      \endaligned\label{eq:RicciTypeIDSfamilygen}
    \end{equation}

    \noindent where
    \begin{align}
      &\tilde{\tilde p}{}_1=c_{v_1}c_{w_1}-c_{v_2}(c_{w_2}+d_{w_2})-(c_{w_1}-c_{v_2})\big((-1)^{n_1-1}\rho^2_{n_1}+(-1)^{n_2-1}\rho^2_{n_2}
      \big),
      \label{eq:srtithma1tildetilde}\\\displaybreak[0]
      &\tilde{\tilde p}{}_2=c_{v_1}(c_{w_1}+d_{w_1})-c_{v_2}c_{w_2}-(c_{v_1}-c_{w_2})\big((-1)^{n_1-1}\rho^1_{n_1}+(-1)^{n_2-1}\rho^1_{n_2}
      \big),
      \label{eq:srtithma2tildetilde}\\\displaybreak[0]
      &\tilde{\tilde p}{}_3=-c_{v_1}d_{w_1}-c_{v_2}d_{w_2}+(d_{w_1}+d_{w_2})\big((-1)^{n_1-1}\rho^5_{n_1}+(-1)^{n_2-1}\rho^5_{n_2}
      \big),
      \label{eq:srtithma3tildetilde}\\\displaybreak[0]
      &\tilde{\tilde q}{}_1=-d_{v_1}(c_{w_1}+d_{w_1})+d_{v_2}d_{w_2}\!-\!(d_{v_1}\!-\!d_{w_2})
      \big((-1)^{\big\lfloor\frac{n_1}2\big\rfloor}\rho^3_{n_1}\!+\!(-1)^{\big\lfloor\frac{n_2}2\big\rfloor}\rho^3_{n_2}
      \big),
      \label{eq:srtithmb1tildetilde}\\\displaybreak[0]
      &\tilde{\tilde
      q}{}_2=-d_{v_1}d_{w_1}-d_{v_2}(c_{w_2}+d_{w_2})-
      (d_{w_1}-d_{v_2})\big((-1)^{\big\lfloor\frac{n_1}2\big\rfloor}\rho^4_{n_1}\!+\!(-1)^{\big\lfloor\frac{n_2}2\big\rfloor}\rho^4_{n_2}
      \big),
      \label{eq:srtithmb2tildetilde}\\\displaybreak[0]
      &\tilde{\tilde q}{}_3=d_{v_1}d_{w_1}+d_{v_2}d_{w_2}-(d_{w_1}+d_{w_2})
      \big((-1)^{\big\lfloor\frac{n_1}2\big\rfloor}\rho^5_{n_1}+(-1)^{\big\lfloor\frac{n_2}2\big\rfloor}\rho^5_{n_2}
      \big),
      \label{eq:srtithmb3tildetilde}\\\displaybreak[0]
      &\aligned
      \tilde{\tilde r}{}_1&=c_{w_1}d_{v_1}-c_{v_2}d_{w_2}+(c_{w_1}-c_{v_2})
      \big((-1)^{\big\lfloor\frac{n_1}2\big\rfloor}\rho^2_{n_1}+(-1)^{\big\lfloor\frac{n_2}2\big\rfloor}\rho^2_{n_2}
      \big)\\&-
      (d_{v_1}-d_{w_2})\big((-1)^{n_1-1}\rho^3_{n_1}+(-1)^{n_2-1}\rho^3_{n_2}\big),
      \endaligned\label{eq:srtithml1tildetilde}\\\displaybreak[0]
      &\aligned
      \tilde{\tilde r}{}_2&=c_{v_1}d_{w_1}-c_{w_2}d_{v_2}+(c_{v_1}-c_{w_2})
      \big((-1)^{\big\lfloor\frac{n_1}2\big\rfloor}\rho^1_{n_1}+(-1)^{\big\lfloor\frac{n_2}2\big\rfloor}\rho^1_{n_2}
      \big)\\&-
      (d_{w_1}-d_{v_2})
      \big((-1)^{n_1-1}\rho^4_{n_1}+(-1)^{n_2-1}\rho^4_{n_2}\big),
      \endaligned\label{eq:srtithml2}
    \end{align}

    \noindent $\rho^z_{0}+\rho^z_{n_1}+\rho^z_{n_2}=1$,
    $z\in\{1,2,3,4,5\}$.\qed
\end{thm}

\begin{thm}[Commutation Formulae Theorem]
  Fifteen of the geometrical objects $a^i_{j\underset{v_1}|m\underset{w_1}|n}-
  a^i_{j\underset{v_2}|n\underset{w_2}|m}$,
  $v_1,v_2,w_1,w_2\in\{0,1,2,3,4\}$, are linearly independent.
\end{thm}

\begin{proof}
  With respect to the Corollary \ref{cor:LinIndependentCovDer} and
  the equation (\ref{eq:linindcovderA}) in this corollary, we get
  \begin{align}
  &a^i_{j\underset{v_1}|m}=x_1a^i_{j\underset1|m}+x_2a^i_{j\underset2|m}+x_3a^i_{j\underset3|m},\label{eq:a|v1(1,2,3)}\\
  &a^i_{j\underset{v_1}|m\underset{w_1}|n}=y_1a^i_{j\underset{v_1}|m\underset1|n}+
  y_2a^i_{j\underset{v_1}|\underset2|n}+y_3a^i_{j\underset{v_1}|m\underset3|n},\label{eq:a|v1|w1(1,2,3)}
  \end{align}

  \noindent for the corresponding scalars $x_1,x_2,x_3,y_1,y_2,y_3$.

  After substituting the expression (\ref{eq:a|v1(1,2,3)}) into the
  equation (\ref{eq:a|v1|w1(1,2,3)}), one gets that the double
  covariant derivative $a^i_{j\underset{v_1}|m\underset{w_1}|n}$ is
  a linear combination of the geometrical objects
  $a^i_{j\underset{v_1'}|m\underset{w_1'}|n}$, for
  $v_1',w_1'\in\{1,2,3\}$.

  In $1-2-3$-Commutation Formulae Theorem \cite{jacovder1},
  it is proved that sixteen of the geometrical objects $a^i_{j\underset{v_1'}|m\underset{w_1'}|n}
  -a^i_{j\underset{v_2'}|n\underset{w_2'}|m}$, for
  $v_1',v_2',w_1',w_2'\in\{1,2,3\}$, are linearly independent, which completes
  the proof for this theorem.
\end{proof}

\section{Identities of Ricci Type with respect to tensor $\hat a$ of
type $(p,q)$}

It holds the next equation.

\begin{scriptsize}
\begin{equation}
  \aligned
  a^{i_1\ldots i_p}_{j_1\ldots j_q\underset v|m\underset w|n}&=
  a^{i_1\ldots i_p}_{j_1\ldots
  j_q|m|n}+c_v\sum_{k=1}^p{L^{i_k}_{\underset\vee{\alpha
  m}}a^{i_1\ldots i_{k-1}\alpha i_{k+1}\ldots i_p}_{j_1\ldots
  j_q|n}}+c_w\sum_{k=1}^p{L^{i_k}_{\underset\vee{\alpha
  n}}a^{i_1\ldots i_{k-1}\alpha i_{k+1}\ldots i_p}_{j_1\ldots
  j_q|m}}\\&+d_v\sum_{l=1}^q{L^\alpha_{\underset\vee{j_lm}}a^{i_1\ldots
  i_p}_{j_1\ldots j_{l-1}\alpha j_{l+1}\ldots j_q|n}}+
  d_w\sum_{l=1}^q{L^\alpha_{\underset\vee{j_ln}}a^{i_1\ldots
  i_p}_{j_1\ldots j_{l-1}\alpha j_{l+1}\ldots
  j_q|m}}+d_wL^\alpha_{\underset\vee{mn}}a^{i_1\ldots
  i_p}_{j_1\ldots j_q|\alpha}\\&+
  c_vc_w\Big\{
  \sum_{k=1}^p{\sum_{\kappa=1}^{k-1}{L^{i_k}_{\underset\vee{\alpha
  m}}L^{i_\kappa}_{\underset\vee{\beta n}}a^{i_1\ldots
  i_{\kappa-1}\beta i_{\kappa+1}\ldots i_{k-1}\alpha i_{k+1}\ldots
  i_p}_{j_1\ldots j_q}}}+
  \sum_{k=1}^p{\sum_{\kappa=k+1}^{p}{L^{i_k}_{\underset\vee{\alpha
  m}}L^{i_\kappa}_{\underset\vee{\beta n}}a^{i_1\ldots i_{k-1}\alpha i_{k+1}\ldots
  i_{\kappa-1}\beta i_{\kappa+1}\ldots
  i_p}_{j_1\ldots j_q}}}
  \Big\}\\&+
  d_vd_w\Big\{
  \sum_{l=1}^q{\sum_{\ell=1}^{s-1}{L^\beta_{\underset\vee{j_\ell
  n}}L^\alpha_{\underset\vee{j_lm}}a^{i_1\ldots i_p}_{j_1\ldots
  j_{\ell-1}\beta j_{\ell+1}\ldots j_{l-1}\alpha j_{l+1}\ldots
  j_q}}}+\sum_{l=1}^q{\sum_{\ell=l+1}^{q}{L^\beta_{\underset\vee{j_\ell
  n}}L^\alpha_{\underset\vee{j_lm}}a^{i_1\ldots i_p}_{j_1\ldots j_{l-1}\alpha j_{l+1}\ldots
  j_{\ell-1}\beta j_{\ell+1}\ldots
  j_q}}}
  \Big\}\\&+
  \sum_{k=1}^p{a^{i_1\ldots i_{k-1}\alpha i_{k+1}\ldots
  i_p}_{j_1\ldots j_q}\Big(
  c_vL^{i_k}_{\underset\vee{\alpha
  m}|n}+c_vc_wL^\beta_{\underset\vee{\alpha
  m}}L^{i_k}_{\underset\vee{\beta
  n}}+c_v(c_w+d_w)L^\beta_{\underset\vee{\alpha
  n}}L^{i_k}_{\underset\vee{\beta
  m}}-c_vd_wL^\beta_{\underset\vee{mn}}L^{i_k}_{\underset\vee{\beta\alpha}}
  \Big)}\\&-
  \sum_{l=1}^q{a^{i_1\ldots i_p}_{j_1\ldots j_{l-1}\alpha
  j_{l+1}\ldots j_q}\Big(
  -d_vL^\alpha_{\underset\vee{j_lm}|n}-d_v(c_w+d_w)L^\beta_{\underset\vee{j_lm}}L^\alpha_{\underset\vee{\beta
  n}}-d_vd_wL^\beta_{\underset\vee{j_ln}}L^\alpha_{\underset\vee{\beta
  m}}+d_vd_wL^\beta_{\underset\vee{mn}}L^\alpha_{\underset\vee{\beta
  j_l}}
  \Big)}\\&+\sum_{k=1}^p{\sum_{l=1}^q{a^{i_1\ldots i_{k-1}\alpha
  i_{k+1}\ldots i_p}_{j_1\ldots j_{l-1}\beta j_{l+1}\ldots j_q}
  \Big(c_wd_vL^\beta_{\underset\vee{j_lm}}L^{i_k}_{\underset\vee{\beta
  n}}+c_vd_wL^\beta_{\underset\vee{j_ln}}L^{i_k}_{\underset\vee{\alpha
  m}}\Big)}}.
  \endaligned\label{eq:apqvw}
\end{equation}
\end{scriptsize}

With respect to the equation (\ref{eq:a=a+cz+dz}), one generalizes
the results obtained in the previous section with the next theorems.

\begin{thm}[General First Ricci-Type Identities Theorem]
  The family of identities of the Ricci Type with respect to a
  non-symmetric affine connection $\nabla$ and a tensor $\hat a$ of
  the type $(p,q)$, $p,q\in\mathbb N$, is

    \begin{equation}
      \aligned
      a^{i_1\ldots i_p}_{j_1\ldots
      j_q\underset{v_1}|m\underset{w_1}|n}-
      a^{i_1\ldots i_p}_{j_1\ldots
      j_q\underset{v_2}|n\underset{w_2}|m}&=\overset1{\mathcal
      S}{}^{i_1\ldots i_p}_{j_1\ldots j_qmn}+\overset2{\mathcal S}{}^{i_1\ldots i_p}_{j_1\ldots j_qmn}
      +(d_{w_1}+d_{w_2})L^\alpha_{\underset\vee{mn}}
      a^{i_1\ldots i_p}_{j_1\ldots j_q|\alpha}
      \\&+\sum_{k=1}^p{a^{i_1\ldots i_{k-1}\alpha i_{k+1}\ldots
      i_p}_{j_1\ldots j_q}\overset1{\mathcal R}{}^{i_k}_{\alpha mn}}-
      \sum_{l=1}^q{a^{i_1\ldots i_p}_{j_1\ldots j_{l-1}\alpha
      j_{l+1}\ldots j_q}\overset 2{\mathcal R}{}^\alpha_{j_lmn}}
      \\&+\sum_{k=1}^p{\sum_{l=1}^q{a^{i_1\ldots i_{k-1}\alpha
      i_{k+1}\ldots i_p}_{j_1\ldots j_{l-1}\beta j_{l+1}\ldots
      j_q}\overset3{\mathcal R}{}^{\beta i_k}_{\alpha j_lmn}}}+\mathcal Z^i_{jmn},
      \endaligned
    \end{equation}

  \noindent where

  \begin{align}
    &\overset1{\mathcal S}{}^{i_1\ldots i_p}_{j_1\ldots j_qmn}=
    \sum_{k=1}^p{\big\{(c_{v_1}\!-\!c_{w_2})L^{i_k}_{\underset\vee{\alpha
      m}}a^{i_1\ldots i_{k-1}\alpha i_{k+1}\ldots i_p}_{j_1\ldots
      j_q|n}\!+\!(c_{w_1}\!-\!c_{v_2})L^{i_k}_{\underset\vee{\alpha
      n}}a^{i_1\ldots i_{k-1}\alpha i_{k+1}\ldots i_p}_{j_1\ldots
      j_q|m}\big\}},\label{eq:RicciTypeIDSS1}\\\displaybreak[0]
      &\overset2{\mathcal S}{}^{i_1\ldots i_p}_{j_1\ldots j_qmn}=
      \sum_{l=1}^q{\big\{(d_{v_1}\!-\!d_{w_2})L^\alpha_{\underset\vee{j_lm}}a^{i_1\ldots
      i_p}_{j_1\ldots j_{l-1}\alpha j_{l+1}\ldots
      j_q|n}
      +(d_{w_1}\!-\!d_{v_2})L^\alpha_{\underset\vee{j_ln}}a^{i_1\ldots
      i_p}_{j_1\ldots j_{r-1}\alpha j_{r+1}\ldots
      j_q|m}\big\}},\label{eq:RicciTypeIDSS2}\\\displaybreak[0]
      &\aligned
      \overset1{\mathcal
      R}{}^i_{jmn}&=R^i_{jmn}+c_{v_1}L^i_{\underset\vee{jm}|n}-c_{v_2}L^i_{\underset\vee{jn}|m}+
      \big[c_{v_1}c_{w_1}-c_{v_2}(c_{w_2}+d_{w_2})\big]L^\alpha_{\underset\vee{jm}}L^i_{\underset\vee{\alpha
      n}}\\&+\big[c_{v_1}(c_{w_1}+d_{w_1})-c_{v_2}c_{w_2}\big]L^\alpha_{\underset\vee{jn}}L^i_{\underset\vee{\alpha
      m}}-(c_{v_1}d_{w_1}+c_{v_2}d_{w_2})L^\alpha_{\underset\vee{mn}}L^i_{\underset\vee{\alpha
      j}},
      \endaligned\label{eq:RicciTypeIDSR1}\\\displaybreak[0]
      &\aligned
      \overset2{\mathcal
      R}{}^i_{jmn}&=R^i_{jmn}-d_{v_1}L^i_{\underset\vee{jm}|n}+d_{v_2}L^i_{\underset\vee{jn}|m}-
      \big[d_{v_1}(c_{w_1}+d_{w_1})-d_{v_2}d_{w_2}\big]L^\alpha_{\underset\vee{jm}}L^i_{\underset\vee{\alpha
      n}}\\&-\big[d_{v_1}d_{w_1}-d_{v_2}(c_{w_2}+d_{w_2})\big]L^\alpha_{\underset\vee{jn}}L^i_{\underset\vee{\alpha
      m}}+(d_{v_1}d_{w_1}+d_{v_2}d_{w_2})L^\alpha_{\underset\vee{mn}}L^i_{\underset\vee{\alpha
      j}},
      \endaligned\label{eq:RicciTypeIDSR2}\\\displaybreak[0]
      &\aligned
      \overset3{\mathcal
      R}{}^{ri}_{sjmn}&=(c_{w_1}d_{v_1}-c_{v_2}d_{w_2})L^r_{\underset\vee{jm}}L^i_{\underset\vee{sn}}+
      (c_{v_1}d_{w_1}-c_{w_2}d_{v_2})L^r_{\underset\vee{jn}}L^i_{\underset\vee{sm}},
      \endaligned\label{eq:RicciTypeIDSR3}
  \end{align}

  \noindent $Z^i_{jmn}$ is the sum of the third and the fourth row
  in the equation \emph{(\ref{eq:apqvw})} and $v_1,v_2,w_1,w_2\in\{0,1,2,3,4\}$.\qed
\end{thm}

\section{Conclusion}

In this article, we generalized the First Ricci-Type Identities
Theorem. It was proved that three of geometrical objects
$a^i_{j|k}$, $a^i_{j\underset1|k}$, $a^i_{j\underset2|k}$,
$a^i_{j\underset3|k}$, $a^i_{j\underset4|k}$ are linearly
independent here.

After that, we generalized the First Ricci-Type Identities Theorem
with respect to a tensor $\hat a$ of the type $(p,q)$,
$p,q\in\mathbb N$.

In the future work, we will generalize the Commutation Formulae
Theorem, $n_1-n_2-n_3$-Second Ricci-Type Identities Theorem and the
$n_1-n_2$-Second Ricci-Type Identities Theorem with respect to
tensors of the types $(p,q)$, $(p,0)$, $(0,q)$.

%Special cases of the last generalization obtained in this paper are
%the families of the Ricci-Type identities with respect to
%contravariant and covariant tensor $\hat a$.

\section{Acknowledgements}

This research is financially supported by Serbian Ministry of
Education, Science and Technological Developments.

The authors thank to the anonymous referee who estimated this paper.

\end{document}